\documentclass[reqno,12pt]{amsart}

\usepackage{epsf}
\usepackage{graphics}
\usepackage{graphicx}
\usepackage{amssymb}
\usepackage{amsmath}

\date{}

\theoremstyle{plain}
\newtheorem{theorem}{Theorem}

\newtheorem{proposition}{Proposition}
\newtheorem{lemma}{Lemma}

\theoremstyle{definition}

\theoremstyle{remark}

\def\N{{\mathbb N}}

\def\R{{\mathbb R}}

\title[The bridge number of arborescent links]{The bridge number of arborescent links \\ with many twigs} 

\author{S.~Baader, R.~Blair, A.~Kjuchukova, F.~Misev}

\begin{document}

\begin{abstract} We prove the meridional rank conjecture for arborescent links associated to plane trees with the following property: all branching points carry a straight branch to at least three leaves. The proof involves an upper bound on the bridge number in terms of the maximal number of link components of the underlying tree, valid for all arborescent links.
\end{abstract}

\maketitle

\thispagestyle{empty}

\section{Introduction}

The family of arborescent tangles can be defined as the minimal family of tangles containing all rational tangles, closed under horizontal and vertical tangle composition~\cite{Co}. Their closures -- arborescent links -- admit a description via weighted plane trees, where each vertex stands for a twisted band, and edges indicate how these bands are glued together. See~\cite{BS, Ga} for a precise definition and Figure~1 for an illustration (ignoring the additional labels and dots in the link diagram for the time being). These descriptions are not unique, since small weights typically allow for simplifications of the underlying tree, without changing the link type.

The meridional rank conjecture by Cappell-Shaneson posits an equality between the bridge number and the meridional rank of a link; see Problem 1.11 in~\cite{Ki}. Early evidence towards this was derived by Boileau and Zimmermann, who showed that two-bridge links are the only links with meridional rank two~\cite{BZ}, and by Rost and Zieschang, who proved the conjecture for torus links~\cite{RZ}.

\begin{figure}[h]
\begin{center}
\raisebox{-0mm}{\includegraphics[scale=1.2]{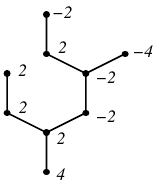}}
\quad
\raisebox{-7mm}{\includegraphics[scale=0.5]{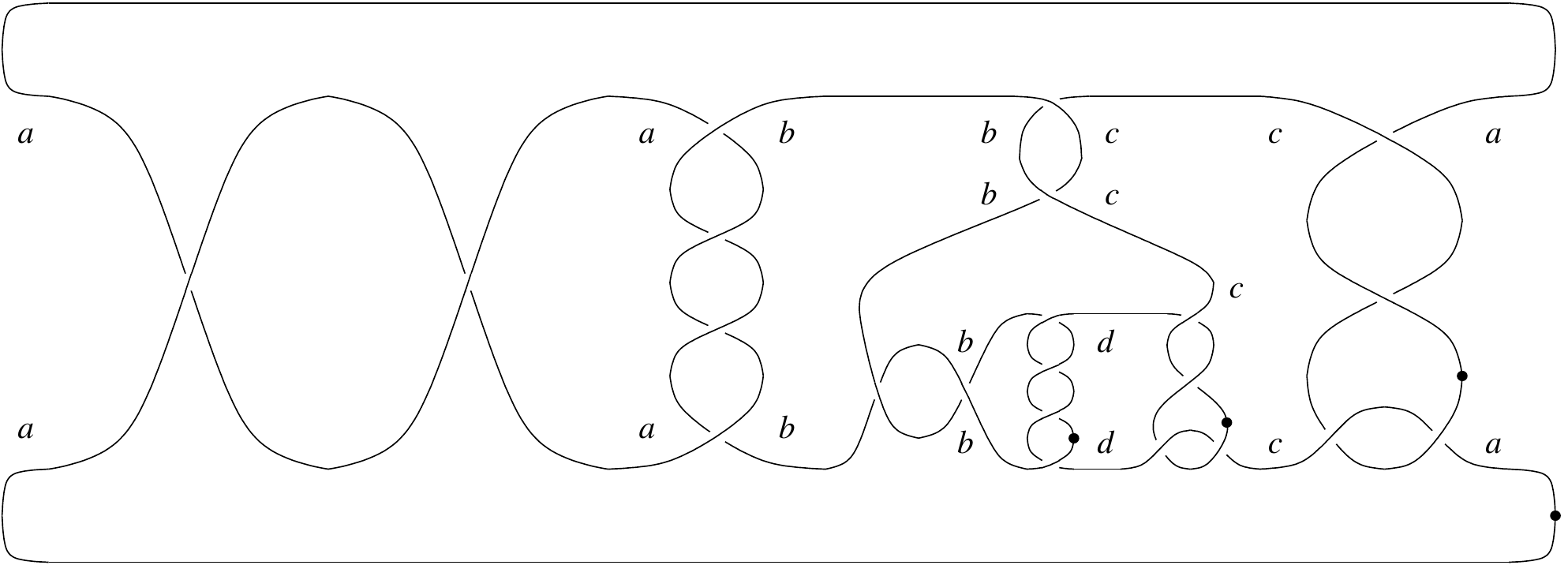}}
\caption{Example of an arborescent knot}
\end{center}
\end{figure}

We recall that the bridge number $\beta(L)$ of a link $L \subset \R^3$ is the minimal number of local maxima of $L$ with respect to a fixed direction, minimised over all isotopic representatives of $L$. The meridional rank $\mu(L)$ is the minimal number of generators of the fundamental group $\pi_1(\R^3 \setminus L)$, where all generators are required to be conjugate to a standard meridional loop of the link $L$. The bridge number of a link is bounded below by its meridional rank.

Given a fixed plane tree $T$, each choice of weights for its vertices determines an arborescent link $L(T)$. 
Define $m(T)$ to be the maximal number of components of $L(T)$ over all links obtained by assigning weights to the vertices of $T$. We will show that $m(T)$ admits a description in terms of the combinatorics of the tree.

\begin{theorem}
	\label{thm:bound}
For every arborescent link $L(T)$ determined by a weighted plane tree $T$, the bridge number of $L(T)$ is bounded above by the maximal component number of $T$:
\[ \beta(L(T)) \leq m(T). \]
\end{theorem}

This bound is sharp for a class of trees, defined next. A twig is a straight branch connecting a leaf to a branching point. A tree $T$ is said to have many twigs if it is obtained from a subtree $T'\subset T$ by adding at least three twigs to every vertex of $T'$.

\begin{theorem} \label{main} Let $L(T)$ be an arborescent link associated to a plane tree $T$ with many twigs and all weights $\neq 0,\pm 1$. The meridional rank conjecture holds for $L(T)$ and
$$\mu(L(T))=\ \beta(L(T))=m(T).$$
\end{theorem} 

To evaluate the maximal component number $m(T)$, we shall use $f(T)$, the flattening number of $T$. Define a subset of edges of $T$ to be flattening if the complement of their interiors is a subforest of $T$ with no vertex of valency bigger than two. The natural number $f(T)$ is the minimal number of edges among all flattening subsets; this definition appears in the context of braid indices of fibred arborescent links in~\cite{Ba}. In the special case of trees with a bipartite ramification structure, where all vertices of valency bigger than two are even distance apart, the number $f(T)$ is easily seen to coincide with the number of leaves of $T$ minus two. The first three authors proved the meridional rank conjecture for links associated to bipartite trees in~\cite{BBK}, establishing Theorems~\ref{thm:bound} and~\ref{main} for this class of links, although the formulation in terms of the maximal component number is new. Trees with many twigs and trees with a bipartite ramification structure have a small intersection, consisting of star-like trees. The links corresponding to these trees are known as Montesinos links.


Our proof is inspired by the technique developed in~\cite{BBK}. We construct Coxeter quotients of the groups $\pi_1(\R^3 \setminus L(T))$ of rank $f(T)+2$ for trees with many twigs. This is done in the next section and establishes the inequality
$$f(T)+2\leq \mu(L(T)).$$

In the third and last section, we show that the bridge number of arborescent links (without restriction) is bounded above by $f(T)+2$,
$$\beta(L(T))\leq f(T)+2.$$
This in turn is done by computing the Wirtinger number of arborescent diagrams, a combinatorial version of the bridge number introduced in~\cite{BKVV}. Moreover, for all plane trees we establish the equality $$m(T)=f(T)+2.$$

\section{Coxeter quotients for arborescent links}

Coxeter groups are encoded by finite simple weighted graphs. Let $\Gamma$ be a finite simple graph with $v(\Gamma)$ vertices, whose edges carry integer weights $\geq 2$. The corresponding Coxeter group $C(\Gamma)$ is generated by $v(\Gamma)$ elements of order two, one for each vertex of $\Gamma$. Every edge with weight $k$ stands for a relation of the form $(st)^k=1$, where $s,t$ is the pair of generators $s,t$ associated with the two vertices of that edge. Elements of $C(\Gamma)$ conjugate to these generators are called reflections. The minimal number of reflections needed to generate $C(\Gamma)$ is called the reflection rank of $C(\Gamma)$; it is known to equal $v(\Gamma)$. The following elementary lower bound for the meridional rank $\mu(L)$ of links in terms of the reflection rank was derived in~\cite{BBK} (Proposition~1, Section~2). 

\begin{proposition} \label{lowerbound}
Let $L$ be a link whose fundamental group surjects onto a Coxeter group $C(\Gamma)$, so that all meridians are mapped to reflections. Then
$\mu(L) \geq v(\Gamma)$. 
\end{proposition}

We will use the term Coxeter quotient for quotients of link groups that arise by sending all meridians of $L$ to reflections of a Coxeter group. These were introduced by Brunner in~\cite{Br}, as homomorphisms onto Artin groups rather than Coxeter groups.

An important class of links that admit non-cyclic Coxeter quotients are two-bridge links,
which can be encoded by rational numbers $\alpha/\beta$ with relatively prime integers $\alpha, \beta$ and $-\alpha<\beta<\alpha$. As explained in~\cite{BBK}, the two-bridge link $L(\alpha/\beta)$ admits a rank two Coxeter quotient generated by two reflections $s,t$ satisfying the relation $(st)^{\alpha}=1$.

The goal of this section is to construct Coxeter quotients of reflection rank $f(T)+2$, for all arborescent links $L(T)$ with the restrictions stated in Theorem~\ref{main}. For this purpose, we need a recursive formula for the flattening number $f(T)$. We say that a tree $\overline{T}$ is obtained from $T$ by adding a ramification point, if $\overline{T}$ contains an edge $e$, whose complement is the union of $T$ and a star-like tree, whose central vertex $c$ is adjacent to $e$. Figure~2 illustrates this operation and serves as a hint for the proof of the following easy fact.

\begin{figure}[h]
\begin{center}
\raisebox{-5mm}{\includegraphics[scale=1.0]{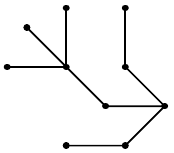}}
$\quad \longrightarrow \quad$
\raisebox{-5mm}{\includegraphics[scale=1.0]{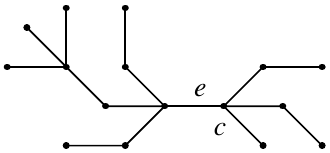}}

\bigskip
\bigskip

\raisebox{-5mm}{\includegraphics[scale=1.0]{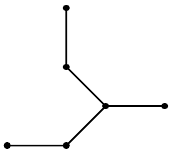}}
$\quad \longrightarrow \quad$
\raisebox{-5mm}{\includegraphics[scale=1.0]{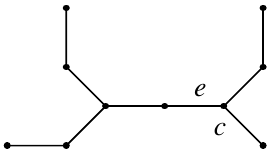}}
\caption{Adding ramification points of valency 4 and 3}
\end{center}
\end{figure}

\begin{lemma} \label{lemma1}
If $\overline{T}$ is obtained from $T$ by adding a ramification point of valency $k\geq 2$, then $f(\overline{T})=f(T)+k-2$.
\end{lemma}

Adding a ramification point of valency $k$ to a tree $T$ has the effect of inserting $k-1$ rational tangles to the arborescent link $L(T)$. This is illustrated in Figure~3 for $k=4$, where each of the three boxes labeled $A,B,C$ stands for a rational tangle determined by the branches incident to the new ramification point, and the number of twists in the central band is given by the weight of that point. Recall that a tree~$T$ satisfying the hypotheses of Theorem~\ref{main} is obtained from a subtree $T' \subset T$ by adding at least three straight branches, or twigs, to every vertex of $T'$. For this reason, $T$ can be constructed inductively from a star-shaped tree by adding ramification points of valency at least four to branching points, i.e. to vertices of valency at least three, as in the upper part of Figure~2.

We will construct a Coxeter quotient of rank $f(T)+2$ by induction on the number of vertices of $T'$. An important element of this construction is that each twist region of the arborescent diagram corresponding to a vertex of $T'$ -- that is, to a branching point of $T$ -- carries a single Coxeter generator. Coincidentally, the base case and the inductive step can be understood in the same diagram: Figure~3 illustrates the base case of a star-shaped tree with three branches, as well as the addition of a ramification point of valency four to an existing branching point. The three labels $x,a,b$ stand for labels of a Coxeter group. Here a label $z$ means that the meridian around the labeled string gets mapped to the generator $z$ of a Coxeter group determined by the link diagram. Our assumption on the weights makes sure that all the rational tangles have non-trivial numerators, hence give rise to Coxeter relations
$(st)^{\alpha}=1$ with $\alpha \geq 2$ (compare the discussion in the second paragraph after Proposition~\ref{lowerbound}).

For the base case, a star-shaped tree, we observe that $L(T)$ admits a Coxeter quotient with $f(T)+2$ generators -- as many as the number of branches. All arcs in the twist region associated with the centre of the star carry the same label,~$x$. This is illustrated on the right side of Figure~3, for a star with three branches.
\begin{figure}[h]
\begin{center} 
\raisebox{-9mm}{\includegraphics[scale=0.8]{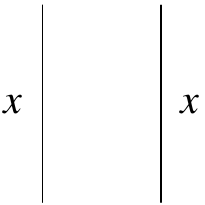}}
$\quad \longrightarrow \quad$
\raisebox{-9mm}{\includegraphics[scale=0.8]{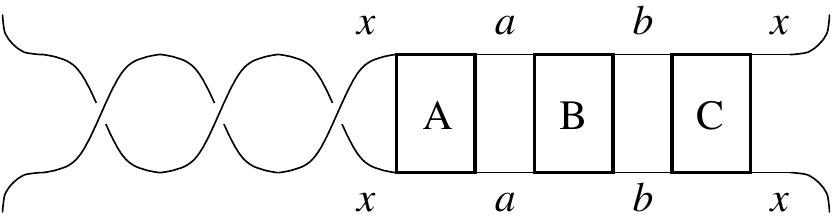}}
\caption{Extending a system of Coxeter generators}
\end{center}
\end{figure}

For the inductive step, we construct a Coxeter quotient of rank $f(\overline{T})+2=f(T)+2+k-2$ for the link $L(\overline{T})$, by adding $k-2$ new reflection generators and $k-1$ new Coxeter type relations determined by the new rational tangles. This is again illustrated in Figure~3 for $k=4$, where the label~$x$ stands for the generator of the branching point, to which we add the new ramification point. The new generators are labeled $a$ and $b$. At this point, it is essential that the new ramification point has at least three twigs. If it had only two twigs, the single new generator $a$ would satisfy two Coxeter type relations with $x$, which are possibly in contradiction. For example, if $p,q \in \N$ are coprime, then the two relations $(ax)^p=1=(ax)^q$ enforce $a=x$.

This inductive construction, together with Proposition~\ref{lowerbound}, proves the desired lower bound on the meridional rank of arborescent links $L(T)$ associated to trees $T$ with many twigs and all weights $\neq 0,\pm 1$:
$$\mu(L(T)) \geq f(T)+2.$$
As pointed out above, the inductive step does not work for arbitrary trees. However, the class of trees $T$ admitting a Coxeter quotient of rank $f(T)+2$ is bigger than the class of trees with many twigs. For example, the labels $a,b,c,d$ of the arborescent knot $L(T)$ in Figure~1 generate a Coxeter quotient of order $f(T)+2=4$, obtained by a similar procedure. The Coxeter relations satisfied by these generators are
$$(ab)^4=(ac)^3=(bc)^2=(bd)^4=(cd)^5=1.$$

\section{Wirtinger and bridge number of arborescent links}

The Wirtinger number of a link, introduced in~\cite{BKVV}, is a combinatorial version of the bridge number. We fix a connected link diagram $D$ with $n$ crossings, whose complement is a union of $n$ arcs. Marking $k$ of these arcs -- called seeds -- by a dot, we obtain a partial coloring of the diagram. We allow the coloring to propagate over crossings by the rule depicted in Figure~4, motivated by the Wirtinger calculus. The idea is that the meridians of all strands marked with a dot are in the subgroup of the link group generated by the meridians of the initially dotted strands, or seeds.

\begin{figure}[h]
\begin{center}
\raisebox{-4mm}{\includegraphics[scale=1.0]{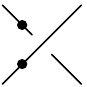}}
$\quad \longrightarrow \quad$
\raisebox{-4mm}{\includegraphics[scale=1.0]{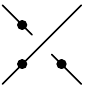}}
\caption{Propagation rule for colors}
\end{center}
\end{figure}

The Wirtinger number $\omega(D)$ is the minimal number of seeds whose coloring propagates to a coloring of the entire diagram~$D$. The main result in~\cite{BKVV} states that the bridge number $\beta(L)$ of a link $L$ coincides with the Wirtinger number of $L$, that is the minimum value of $\omega(D)$ among all diagrams of $L$. In particular, the Wirtinger number of any diagram $D$ of a link $L$ is an upper bound for the bridge number: $$\beta(L) \leq \omega(D).$$

In this section, we will prove that a suitable choice of $f(T)+2$ seeds in a diagram of the arborescent link $L(T)$ propagates to a coloring of the entire diagram, without any restriction on the tree $T$ and its weights. This implies that for all arborescent links $L(T)$
$$\mu(L(T)) \leq \beta(L(T)) \leq f(T)+2.$$
Combining this with the inequality of the previous section, $\mu(L(T)) \geq f(T)+2$, valid for all arborescent links $L(T)$ with the restrictions stated in Theorem~\ref{main}, we obtain the two desired equalities:
$$\beta(L(T))=\mu(L(T))=f(T)+2.$$

We are left to construct diagram colorings for $L(T)$ with $f(T)+2$ seeds, starting with the base case $f(T)=0$, or trees $T$ without ramification points. The corresponding links $L(T)$ are precisely the two-bridge links, or closures of rational tangles. As we can see from Figure~5, two suitably chosen initial seeds are enough to propagate to a coloring of the entire rational tangle. Note that no information on over- and undercrossings is needed in these diagrams.

\begin{figure}[h]
\begin{center}
\raisebox{-0mm}{\includegraphics[scale=0.8]{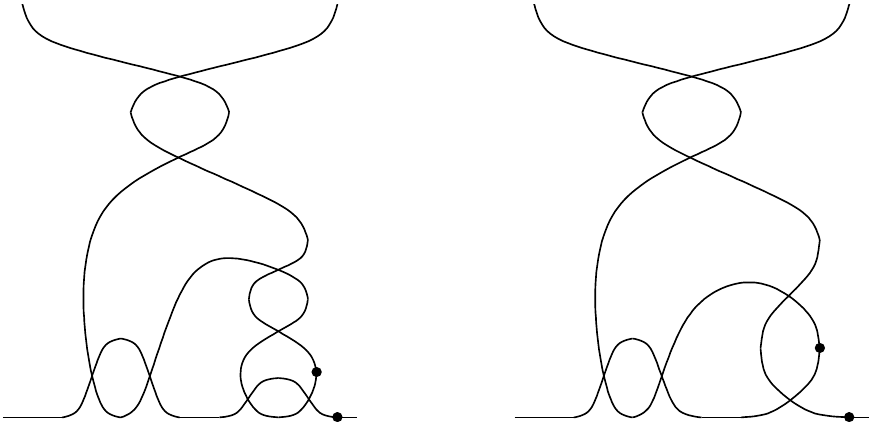}}
\caption{Initial seeds for rational tangles}
\end{center}
\end{figure}

Even more is true: a single seed on any of the four outgoing strings of a rational tangle can be complemented by a second seed, so that the coloring propagates to a coloring of the entire tangle. This may require a sequence of flype moves, as shown in Figure~6.

\begin{figure}[h]
\begin{center}
\raisebox{-20mm}{\includegraphics[scale=0.8]{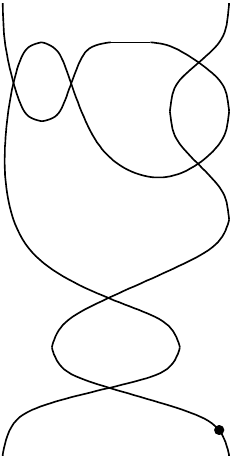}}
$\qquad \longrightarrow \qquad$
\raisebox{-20mm}{\includegraphics[scale=0.8]{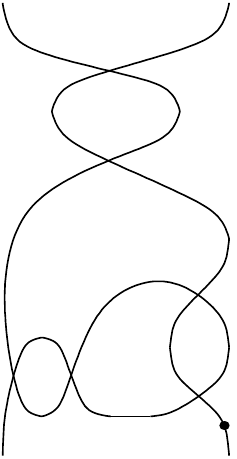}}
\caption{Flype}
\end{center}
\end{figure}

\begin{figure}[h] \label{seeds}
\begin{center}
\raisebox{-0mm}{\includegraphics[scale=0.8]{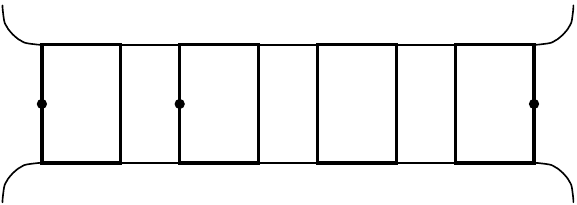}}
\caption{Seed extension}
\end{center}
\end{figure}

We are now set for an inductive construction of a diagram coloring for $L(T)$ with $f(T)+2$ initial seeds, making once again use of Lemma~\ref{lemma1}. Let $\overline{T}$ be a tree obtained from $T$ by adding a ramification point of valency $k$ to $T$. Suppose that a link diagram of $L(T)$ admits a coloring by $f(T)+2$ initial seeds that propagate to a coloring of the entire link diagram. Note that this step potentially requires applying flypes to the standard diagram of $L(T)$, all of which are supported in rational tangles corresponding to branches of T. We obtain a diagram of $L(\overline{T})$ by adding a twist region and $k-1$ rational tangles to the (possibly flyped) diagram of $L(T)$. The twist region corresponding to the new ramification point can be included in any of these rational tangles. For example, in Figure~3, the three crossings together with tangle $A$ form a single rational tangle. Adding $k-2$ suitable seeds, one for each rational tangle except one, as shown schematically in Figure~7 for $k=5$, gives rise to a set of $f(T)+k=f(\overline{T})+2$ seeds that propagate to a coloring of a link diagram of $L(\overline{T})$. This completes the proof of the inequality $\beta(L(T))\leq f(T)+2$. 

To obtain Theorems~\ref{main} and~\ref{thm:bound}, it remains to show that $m(T)=f(T)+2$ for all plane trees. Once again we consider star-like trees as a first step. 

\begin{lemma} \label{lem:starlike}
Given a star-like plane tree $T$, there exists a vertex labeling of $T$ such that the corresponding link $L(T)$ has $f(T)+2$ components.
\end{lemma}
\begin{proof}
Suppose the center vertex of $T$ has valency $n\geq 2$. Then $f(T)=n-2$ and $L(T)$ is a Montesinos link on $n$ rational tangles as pictured.
\begin{figure}[ht]
\begin{center}
\includegraphics[scale=1]{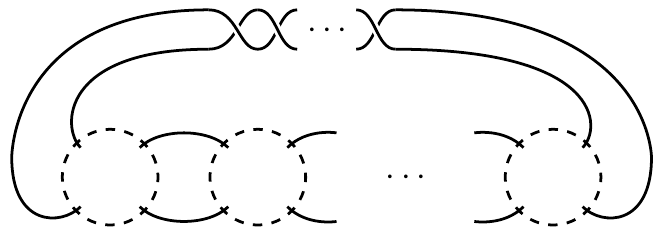}
\caption{Montesinos link on $n$ rational tangles.}
\end{center}
\end{figure}
It now suffices to show that we can choose the labels of $T$ to achieve the connectedness diagram given in Figure~9, so that $L(T)$ has $n=f(T)+2$ components.
\begin{figure}[ht]
\begin{center}
\includegraphics[scale=1]{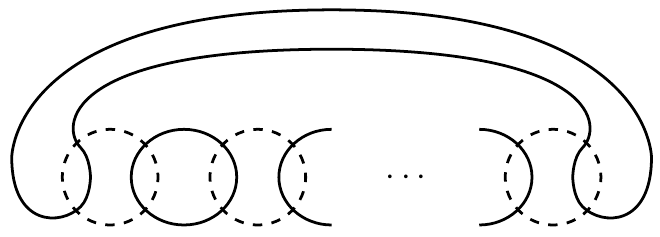}
\caption{Connectedness diagram for a Montesinos link on $n$ rational tangles with suitably chosen weights.}
\end{center}
\end{figure}
Let $R$ be any of the $n$ rational tangles. If $R$ has length one, that is, it consists of a single twist region, then we may choose any even label to achieve the desired connectedness. If $R$ has length at least two, all three ways to connect the four endpoints in pairs can be achieved.
\end{proof}

\begin{proposition}\label{m=f+2}
For any plane tree $T$, the flattening number $f(T)$ and the maximal component number $m(T)$ are related as follows.
\[ m(T)=f(T)+2. \]
\end{proposition}
\begin{proof}
Assign weights to $T$ so that $L(T)$ is a link realizing the maximal component number $m(T)$. We then have $$f(T)+2 \geq \beta(L(T)) \geq m(T),$$ since the inequality $f(T)+2 \geq \beta(L(T))$ holds for any choice of weights, and the number of components of the link $L(T)$ can not exceed its bridge number. 

It remains to show that   $f(T)+2 \leq m(T)$. We do this by choosing labelings on $T$ that result in a link $L(T)$ with $f(T)+2$ components. We induct on the number of ramifications needed to construct $T$. The base case, $T$ is star-like, is treated in Lemma~\ref{lem:starlike}. Now assume $T$ is the result of adding a ramification point, that is, a star-like tree $T_2$, to an arbitrary plane tree $T_1$.
\begin{figure}[ht]
\begin{center}
\includegraphics[scale=1]{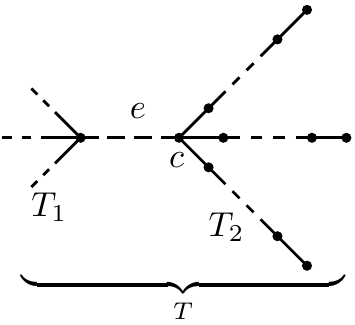}
\caption{New ramification point, $c$.}
\end{center}
\end{figure}
Note that $T_1$ can be constructed using strictly fewer ramifications than $T$. By induction, there exists a labeling of $T_i$ such that $L(T_i)$ has $f(T_i)+2$ components, for $i=1,2$. These labels of $T_1$ and $T_2$ induce a labeling of $T$ which results in the connectedness diagram for $L(T)$ pictured below. 
\begin{figure}[ht]
\begin{center}
\includegraphics[scale=1]{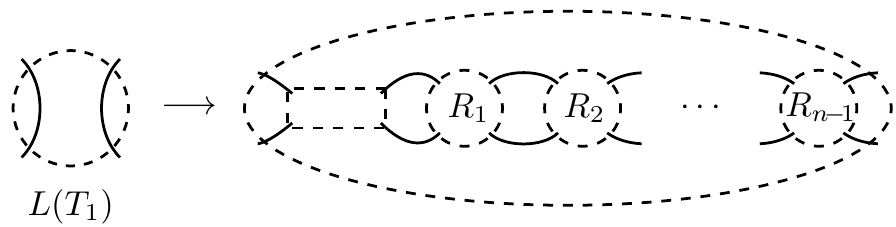}
\caption{Tangle substitution corresponding to the new ramification point.}
\end{center}
\end{figure}
Here, the $R_i$ are rational tangles and the ramification point has valence $n$. 
\begin{figure}[ht]
\begin{center}
\includegraphics[scale=1]{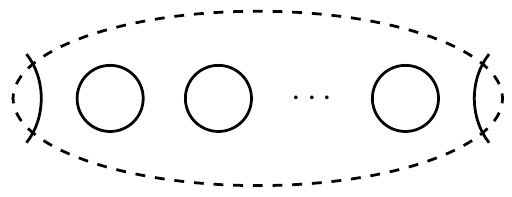}
\caption{Connectedness diagram of inserted tangle induced by the choice of labels on $T_2$.}
\end{center}
\end{figure}

Figure~12 shows the connectedness diagram induced by the chosen labels for $T_2$, with $n-2$ components of $L(T)$ contained therein. Hence, the number of components of $L(T)$ is $$(f(T_1)+2)+(n-2) = f(T_1) +n.$$ By Lemma~\ref{lemma1}, $f(T)=f(T_1)+n-2$, so, as claimed, $L(T)$ has $f(T)+2$ components.
\end{proof}


\begin{thebibliography}{9}

\bibitem{Ba}
     S.~Baader: \emph{Hopf plumbing and minimal diagrams},
Comment. Math. Helv.~80 (2005), no.~3, 631--642. 

\bibitem{BBK}
     S.~Baader, R.~Blair and A.~Kjuchukova: \emph{Coxeter groups and meridional rank of links}, arXiv:1907.02982.  

\bibitem{BKVV}
     R.~Blair, A.~Kjuchukova, R.~Velazquez, and P.~Villanueva: \emph{Wirtinger systems of generators of knot groups}, to appear in Communications in Analysis and Geometry.

\bibitem{BZ}
     M.~Boileau and B.~Zimmermann: \emph{The $\pi$-orbifold group of a link}, Math. Z.~200 (1989), no.~2, 187--208. 

\bibitem{BS}
     F.~Bonahon and L.~C.~Siebenmann: \emph{New Geometric Splittings of Classical Knots
and the Classification and Symmetries of Arborescent Knots}, available at
https://dornsife.usc.edu/francis-bonahon/publications/

\bibitem{Br}
     A.~M.~Brunner: \emph{Geometric quotients of link groups}, Topology and its Applications~48 (1992), no.~3, 245--262.

\bibitem{Co}
     J.~H.~Conway: \emph{An enumeration of knots and links, and some of their algebraic properties}, Computational problems in abstract algebra, Elsevier, 1970, pp. 329--358.

\bibitem{Ga}
     D.~Gabai: \emph{Genera of the arborescent links}, Mem. Amer. Math. Soc.~59 (1986), no.~339, i--viii and 1--98. 

\bibitem{Ki}
     R.~Kirby: \emph{Problems in low-dimensional topology}, Proceedings of Georgia Topology Conference, Part~2 (R. Kirby, ed.), Citeseer, 1995.

\bibitem{RZ}     
	M.~Rost and H.~Zieschang: \emph{Meridional generators and plat presentations of torus links}, J. London Math. Soc.~2 (1987), no.~3,551-562. 

\end{thebibliography}
\end{document}